\documentclass[a4paper,10pt]{article}
\usepackage[utf8]{inputenc}
\usepackage{amsmath}
\usepackage{amsfonts}
\usepackage{stmaryrd}
\usepackage{amssymb}
\usepackage{amsthm}
\usepackage{graphicx}
\usepackage{xspace}
\usepackage{cite}
\usepackage[margin=1.32in]{geometry}
\usepackage{etoolbox}
\usepackage{xcolor}
\usepackage{comment}
\DeclareMathAlphabet{\mathpzc}{OT1}{pzc}{m}{it}

\newtheorem{theorem}{Theorem}[section]
\newtheorem{lemma}[theorem]{Lemma}
\newtheorem{proposition}[theorem]{Proposition}
\newtheorem{corollary}[theorem]{Corollary}
\newtheorem{definition}[theorem]{Definition}

\theoremstyle{definition}
\newtheorem{remark}[theorem]{Remark}

\newcommand{\al}{\alpha}
\newcommand{\be}{\beta}

\newcommand{\de}{\delta}
\newcommand{\e}{\epsilon}

\newcommand{\f}{\frac}

\newcommand{\lra}{\longrightarrow}

\newcommand{\ra}{\rightarrow}

\newcommand{\sbq}{\subseteq}
\newcommand{\spq}{\supseteq}

\newcommand{\x}{\times}
\newcommand{\z}{\zeta}

\newcommand{\SL}{{\mathcal{L}}}

\newcommand{\bbn}{{\Bbb N}}

\newcommand{\bbr}{{\Bbb R}}

\newcommand{\supp}{\textup{\textrm{supp}}\xspace}
\newcommand{\diag}{\textup{\textrm{diag}}\xspace}
\newcommand{\trace}{\textup{\textrm{trace}}\xspace}
\newcommand{\Ad}{\textup{\textrm{Ad}}\xspace}

\newcommand{\Mat}{\textup{\textrm{Mat}}\xspace}

\newcommand{\rspan}{\textup{$\mathbb{R}$\textrm{-span}}\xspace}

\newcommand{\Sl}{\textup{\textrm{SL}}\xspace}
\newcommand{\Gl}{\textup{\textrm{GL}}\xspace}

\usepackage[pdftex,pagebackref,letterpaper=true,colorlinks=true,pdfpagemode=none,urlcolor=blue,linkcolor=blue,citecolor=blue,pdfstartview=FitH]{hyperref}
\usepackage[english]{babel}

\newtoggle{arxive}
\togglefalse{arxive}

\newtoggle{paper}
\togglefalse{paper}

\newtoggle{thesis}
\toggletrue{thesis}

\makeatletter
\renewcommand*{\@fnsymbol}[1]{\ensuremath{\ifcase#1\or  \dagger\or \ddagger\or
   \mathsection\or \mathparagraph\or \|\or **\or \dagger\dagger
   \or \ddagger\ddagger \else\@ctrerr\fi}}
\makeatother
\title{Random Walks on Homogeneous Spaces by Sparse Solvable Measures}
\author{C. Davis Buenger
\thanks{This material is based upon work supported by the National Science Foundation under Grant No. 0932078 000 while the author was residence at the Mathematical Science Research Institute in Berkeley, California, during the Spring 2015 semester.}}

\begin{document}

\allowdisplaybreaks
\maketitle

\begin{abstract} The paper analyzes a specific class of random walks on quotients of $X:=\Sl(k,\bbr)/
\Gamma$ for a lattice $\Gamma$. Consider a one parameter diagonal subgroup, $\{g_t\}$, with an associated
abelian expanding horosphere, $U\cong \bbr^k$,  and let $\phi:[0,1]\ra U$ be 
a sufficiently smooth curve satisfying the condition that that the derivative of $\phi$ spends $0$ time in any one subspace of $\bbr^k$. Let $
\mu_U$ be the measure defined as $\phi_*\lambda_{[0,1]},$ where $\lambda_{[0,1]}$ is the Lebesgue measure 
on $[0,1]$. Let $\mu_A$ be a measure on the full diagonal subgroup of $\Sl(k,\bbr)$, such that almost surely 
the random walk on the diagonal subgroup $A$ with respect to this measure grows exponentially in 
the direction of the cone expanding $U$. Then the random walk starting at any point $z\in X$, and alternating steps given by $\mu_U$ and $\mu_A$ equidistributes 
respect to $\Sl(k,\bbr)$-invariant measure on $X$. Furthermore, the measure defined by $\mu_A*\mu_U*\dots*\mu_A*
\mu_U*\delta_z$ converges exponentially fast to the $\Sl(k,\bbr)$-invariant measure on $X$. 
\end{abstract}
\section{Introduction}

Here we continue the study of random walks on homogeneous spaces initiated by A. Eskin and G. Margulis \cite{EM} and furthered by Y. Benoist and J.-F. Quint \cite{BQ}, \cite{BQ2}, \& \cite{BQ3}. Eskin and Margulis  considered the case where $G$ was a Lie group, $\Gamma$ a lattice in $G$, and $\mu$ a probability measure with finite moments on $G$  such that  the group 
  generated by the support of $\mu$ was semisimple. Under these conditions, Eskin and Margulis  show that the random walk with
   steps given by $\mu$ starting at any point $x\in X:=G/\Gamma$ is non-divergent. To be precise,  let $\delta_x$ be the 
   Dirac measure at $x$, for a measure $\mu$ let $\mu^{*n}$ denote the $n$-fold convolution of $\mu$.  For a measure $\mu$ on $G$ let $H_\mu$ denote the noncompact part  of the Zariski closure of the support of $\mu$. With this notation, Eskin and Margulis  showed the following:
\begin{theorem}[\cite{EM} Theorem 2.1] Suppose $H_\mu$ is semisimple, and for all $g\in G$, $gH_\mu g^{-1}$ is not contained in any proper $\Gamma$-rational parabolic subgroup of $G$. Then for every compact set $C\subseteq X$, there exists a compact set $K\supseteq C$ such that for every $x \in C$ and every $n>0$,
$$\mu^{*n}*\delta_x(K)\geq 1-\epsilon.$$
\end{theorem}   
 Benoist and Quint showed equidistribution in a homogeneous subspaces under stronger assumptions.
 \begin{theorem}[Theorem 1.1 \& 1.2 \cite{BQ3}]\label{BQT} Let $G$ be real Lie group with Lie algebra $\mathfrak{g}$, $\Gamma$ be a lattice in $G$, 
 and $\Lambda$ be a compactly generated sub-semigroup of $G$. 
  Assume that the Zariski closure of $\Ad(\Lambda)\subseteq \Gl(\mathfrak{g})$ is semisimple with no compact factors. Then
 \begin{enumerate}
 \item[(a)] For every $x$ in $G/\Gamma$, there exists a closed subgroup $H\spq\Lambda$ of $G$ such that $\overline{\Lambda x}=Hx$ and $Hx$ carries an $H$-invariante probability measure $\nu_x$.
\item[(b)] If $\mu$ is a compactly supported Borel probability measure on $\Lambda$ whose support spans a dense sub-semigroup of $\Lambda$, then 
$$\f{1}{n}\sum_{k=0}^{n-1}\mu^{*k}*\delta_x\lra\nu_x.$$
\item[(c)] More precisely, if $g_1,\dots,g_n,\dots$ is a sequence of independent identically distributed random elements of $\Lambda$ with Law $\mu$, then almost surely,
$$\f1n\sum_{k=0}^{n-1}\delta_{g_k\dots g_1x}\lra\nu_x.$$
 \end{enumerate}
 \end{theorem}

I consider the case when the measure $\mu$ is supported on a specific sparse subset of a solvable subgroup and  show exponential convergence of $\mu^{*n}*\delta_x$  to the $G$-invariant probability measure on $G/\Gamma$. 
For the course of the note fix $k_1,k_2\in\bbn$. Let $k_0=k_1+k_2$, $k=k_1\cdot k_2$,  and $G=\Sl(k_0,\mathbb{R})$. Let $\Gamma$ be an arbitrary lattice in $G$, $X=G/\Gamma$, and
$$U:=\left\{\left[\begin{array}{cc}I_{k_1}& M\\0&I_{k_2}\end{array}\right]:M\in Mat_{k_1\times k_2}\right\}.$$
Let $\mathpzc{u}:\mathbb{R}^{k}
\rightarrow U$ be any isomorphism of groups.
For $n\in\bbn$ let 
$A_n$ be the group of diagonal subgroup of $\Sl(n,\bbr)$. 
 For $a=\diag(d_1,\dots,d_n)\in A_n$ and $i=1,\dots,n$, let $\pi_i(a)=d_i$. Throughout these notes for a Lie group $F$ let $\lambda_F$ denote a $F$-invariant  measure for $F$. If $\Delta$  is a lattice in $F$, then let $\lambda_{F/\Delta}$ denote the $F$-invariant probability measure on the homogeneous space $F/\Delta$.   If $E$ is a finite volume measurable subset of $F$ or $F/\Delta$ with respect to a $F$-invariant measure, then let $\lambda_E$ denote the normalized $F$-invariant measure restricted to $E$. 
 
\begin{definition} Let $\psi:[0,1]\ra\bbr^n$. We say that $\psi$ is totally non-planar if $f\in C^1$ and for every $v\in\bbr^n$,
$$\lambda_{[0,1]}\left(\left\{s:\psi'(s)\text{ exists and }\psi'(s)\cdot v=0\right\}\right)=0.$$
\end{definition}

 
For example, if $\psi:[0,1]\ra\bbr^n$ is analytic  and not contained in a proper affine subspace, then $\psi$ is totally non-planar. As a point of comparison, Kleinbock, Lindenstrauss, and Weiss \cite{KLW} define a measure $\nu$ on $\bbr^n$ to be non-planar if $\nu(\SL)$ for every affine hyperplane $\SL$ of $\bbr^n$. Our techniques require a stronger assumption on the support of $\mu_U$.


 \begin{definition} Let $\mu$ be a compactly supported measure on $A_{k_0}$, for $i=1,\dots, k_0$. For $i=1,\dots,k_0$ let  let $x_i$ be a random variable with law $\ln({\pi_i}_*\mu)$, and let
$$\al_i:=E x_i.$$
 We say that $\mu$ is asymptotically $U$-expanding if  for all $i\in\{1,\dots,k_1\}$ and $j\in\{i=k_1+1,\dots,k_0\}$, we have $\al_i-\al_j>0$. 

 \end{definition}
 
  For an asymptotically $U$-expanding measure $\mu$, there exist a constant $C>1$ and $\eta>0$ such that for all $u\in U$, $$\mu^{*n}\left(\left\{a\in A_{k_0}:\f{\|aua^{-1}\|}{\|u\|}>C^n\right\}\right)>1-e^{-\eta n},$$
  where $\|\cdot\|$ is the standard Euclidian norm on $U\cong\bbr^{k}$.  
  See Corollary \ref{growth} for a more precise statement.

\begin{theorem}\label{mainthm} There exists $m\in\bbn$ depending only on $G$ such that the following holds. Let $\phi:[0,1]\rightarrow\mathbb{R}^{k}$ be a totally non-planar $C^m$ function.  Let $\mu_U=\mathpzc{u}_*\phi_*\lambda_{[0,1]}$. Let $\mu_A$ be a compactly supported asymptotically $U$-expanding probability 
measure on $A_{k_0}$, and let
 Let $\mu=\mu_A*\mu_U$.

Then there exists $\eta>0$ such that for any compact set $L\subseteq X$, and $f\in C^\infty_{c}(X)$, there exists  a constant $C:=C(f,L)$ such that for all $n\in\mathbb{N}$ and all $x\in L$,
\begin{equation}\label{eqn1}\left|\int f\,d(\mu^{*n}*\delta_x)-\int f\,d\lambda_{X}\right|<Ce^{-\eta n}.\end{equation}
\end{theorem} The choice of $m\in\bbn$ in this theorem arrises from the mixing results of \cite{KM}. Indeed, we employ exponential mixing as the primary tool in the paper and its application requires sufficiently many derivatives of a function on the expanding leaf. See section \ref{sec4} for a more detailed explanation. Furthermore,
this theorem implies the following theorem analogous to Theorem \ref{BQT}  part (c) proved by Benoist and Quint.
\begin{corollary}\label{mainthm3}Let $G,\Gamma, \mu$ be as in Theorem \ref{mainthm}. Let $g_1,g_2,\dots $ be i.i.d. random variables on $G$ with law $\mu$. Then for any $z\in X$ and almost every $(g_1,g_2,\dots)$ in $G^\bbn$ with respect to $\mu^{\otimes\bbn}$
$$\f1n\sum_{k=0}^{n-1}\delta_{g_k\dots g_1x}\lra \lambda_X.$$
\end{corollary}
For $l\in\bbn$, let 
$$A_{l,\epsilon}:=\left\{\diag(e^{t_1},\dots,e^{t_{l}})\in A_{l}:|t_i|<\epsilon \,\forall\,i\in\{1,\dots,l\}\right\}.$$
Theorem \ref{mainthm} requires $\mu_U$ to be the parameter measure of a curve, but more general measures are possible. By making stronger assumptions on the measure $\mu_A$, one can have significantly greater flexibility choosing $\mu_U$.
\begin{theorem}\label{mainthm2} There exists $m\in\bbn$ depending only on $G$ such that the following holds. Let $\phi:[0,1]\rightarrow\mathbb{R}^{k}$ be a piecewise $C^m$ function and there exist points $x_1,\dots,x_{k_0}\in[0,1]$  such that $\mathbb{R}$-span$\{\phi'(x_1),\dots,\phi'(x_{k}) \}=\mathbb{R}^{k}$. Let $0<c<1$ and $\mu_R$ be any compactly supported probability measure on $\bbr^k$. Let $\mu_U=\mathpzc{u}_*(c\mu_R+(1-c)\phi_*\lambda_{[0,1]})$. Then there exists $\e>0$ such that if 
$\mu_A$ is asymptotically $U$-expanding probability measure supported on $A_{k_0,\epsilon}$ and $\mu=\mu_A*\mu_U$,  then there exists $\eta>0$ such that for any compact set $L\subseteq G/\Gamma$, and $f\in C^\infty_{\text{comp}}(G/\Gamma)$, there exists  a constant $C:=C(f,L)$ such that for all $n\in\mathbb{N}$ and all $x\in L$,
\begin{equation}\label{eqn2}\left|\int f\,d(\mu^{*n}*\delta_x)-\int f\,d\lambda_{X}\right|<Ce^{-\eta n}.\end{equation}
\end{theorem}
As a final point of context, let us compare the results of this paper to the recent results of A. Eskin and E. Lindenstrauss. In a presentation at the MSRI \cite{EL}, Eskin considered a Lie group $G$ a lattice $\Gamma$ and defined a measure $\mu$ to be {\em uniformly expanding measure} if there exist a constant $C>0$ and an integer $n$ such that for all $v\in \mathfrak{g}$ 
\begin{equation}\label{unifex1}\int_{G}\f{\log(\|\Ad(g)v\|)}{\|v\|}d \mu^{*n}>C.\end{equation}
Equivalently, $\mu$ is uniformly expanding  if and only if  for every $v\in\mathfrak{g}$ and a.e. infinite word ${\bf g}=(g_1,g_2,\dots)$ with respect to $\mu^\bbn$ we have
\begin{equation}\label{unifex2}\lim_{n\ra\infty} \f{1}{n}\log\|\Ad(g_n\dots g_1)v\|>0.\end{equation}
With this set up they show:
\begin{theorem}[\cite{EL}]
Suppose $\mu$ is a compactly supported uniformly expanding measure. Then any $\mu$-stationary measure on $G/\Gamma$,   is invariant under the group generated by the support of $\mu$. 
\end{theorem}
 For example, 
if $\Lambda$ is the support of  a measure $\mu$ on $G$ and $\overline{\Ad\Lambda}^Z$ is simisimple with no center and no compact factors, then  $\mu$ is uniformly expanding. Thus the measures considered by Benoist and Quant are uniformly expanding. 
Furthermore, the measures I consider are also uniformly expanding (See proposition \ref{ue}).

However, my results do not follow directly from theirs. Classification of stationary measures is an inherently weaker problem than proving that the measures $\mu^{*n}*\delta_x$ converge to the $G$-invariant probability  measure on $X$. To prove the latter question from the results of Eskin and Lindenstrauss, one must additionally show that  a limit probability measure exists and that invariance under the support of $\mu$ implies  $G$-invariance. Even after this though, one can only show weak-$*$ convergence of $\mu^{*n}*\delta_x$ to the  $G$-invariant probability measure. We further show that this weak-$*$ convergence occurs exponentially.  

{\bf Acknowledgements:} Thank you to the many people who helped me on this project including  Nimish Shah, Barak Weiss, Yves Benoist, Jean-Fran{\c{c}}ois Quint, and Ronggang Shi. Furthermore, let me especially thank Ilya Khayutin for our numerous constructive conversations. Finally, let me thank the MSRI for providing a fabulous work space.
\section{Push forward of Lebesgue measure by sums of totally non-planar functions}\label{non-planar}
In this section, we develop techniques to produce absolutely continuous measures with respect to Lebesgue measure on $\bbr^k$ by convolving parameter measures on $k$ totally non-planar curves. In general, the production of an absolutely continuous measure follows from the inverse function theorem and properties of totally non-planar functions (a sum of $k$ totally non-planar functions with independent variables is invertible almost everywhere in its image). However, we take special care in this section to control the $C^m$ norms of the Radon-Nikodym  derivative of the measure. 
\begin{lemma}\label{span} Let $\psi_1,\dots,\psi_k:[0,1]\ra\bbr^k$ be  totally non-planar functions.  Then 
$$\lambda_{[0,1]^k}\left(\left\{x\in[0,1]^k: \det\left[\psi'_1(x_1),\dots,\psi_k'(x_k)\right]=0\right\}\right)=0.$$
\end{lemma}
\begin{proof}
We will show the stronger claim that for $i=1,\dots,k$  almost every ${\bf y}\in[0,1]^i$  with respect to $\lambda_{[0,1]^i}$ the subspace $\rspan\{
\psi_1'(y_1),\dots,\psi_i'(y_i)\}$ is $i$ dimensional. The $i=1$ case follows trivially. Assume that the claim holds for $i<k$.  Let ${\bf y}\in[0,1]^{i+1}$ be such that $
\rspan\{\psi_1'(y_1),\dots,\psi_i'(y_i)\}$ is an $i$ dimensional subspace of $\bbr^k$. Then by the fact that $\psi$ is 
totally non-planar the probability that $y_{i+1}$ is outside of this subspace $\rspan\{\psi_1'(y_1),\dots,\psi_i'(y_i)\}$ is 1. 
Thus the dimension of the subspace $\rspan\{\psi_1'(y_1),\dots,\psi_{i+1}'(y_{i+1})\}$ is $i+1$ dimensional with 
probability 1.
\end{proof}
\begin{lemma}\label{nonplan}Let $\psi:[0,1]\ra\bbr^k$ be a totally non-planar function, and $M\in\Gl(k,\bbr)$, then $M\psi$ is a totally non-planar function.
\end{lemma}
\begin{proof}This follows easily since $\f{d}{dt}(M\psi(t))=M\psi'(t)$.
\end{proof}

\begin{corollary}\label{abscon} Let $\psi_1,\dots,\psi_k:[0,1]\ra\bbr^k$ be  totally non-planar functions, and $M_1,\dots,M_k\in\Gl(k,\bbr)$. Define $\Psi:[0,1]^k\ra\bbr^k$ as
$$\Psi(t_1,\dots,t_k)=\sum_{i=1}^kM_i\psi_i(t_i).$$
Then $\Psi_*\lambda_{[0,1]^k}$ is absolutely  continuous with respect to Lebesgue measure on $\bbr^k$.
\end{corollary}
\begin{proof}
By Lemmas \ref{span} and \ref{nonplan}, the map $\Psi$ is locally invertible. Hence by the inverse function theorem  $\Psi_*\lambda_{[0,1]^k}$ is absolutely  continuous with respect to Lebesgue measure on $\bbr^k$.
\end{proof}
As mentioned above, Corollary \ref{abscon} does not suffice in proving our results. With slightly stronger assumptions on $\mu_A$, one could use Corollary \ref{abscon} combined with the mixing of the expanding horispherical leaf  under the action of $\diag(e^{k_2},\dots,e^{k_2},e^{-k_1},\dots,e^{-k_1})$ to prove  weak-$*$ convergence of $\mu^{*n}$ to the $G$-invariant probability measure on $X$. However, we are interested in the rate of this convergence and to apply the mixing results of Kleinbock and Margulis (See Theorem \ref{mix})  we need to control the derivatives of order up $m$ for a fixed $m$ depending on $G$, of the Radon Nikodym derivatives of the measure on the expanding leaf\cite{KM}. Sections \ref{non-planar} and \ref{three} set up the necessary structure to control these derivatives.
 
For the remainder of the section, fix a totally non-planar function $\psi:[0,1]\ra\bbr^k$ and define the function $F_\psi:A_k^k\times[0,1]^k\ra \bbr$ by
 $$(a_1,\dots,a_k,x_1,\dots,x_k)\mapsto \det\left[a_1\psi'(x_1),\dots,a_k\psi(x_k)\right].$$
Then by Lemmas \ref{span} and \ref{nonplan} for every ${\bf a}=(a_1,\dots,a_k)\in A_k^k$,
$$\lambda_{[0,1]^k}\left(\left\{x\in[0,1]^k:F_\psi({\bf a},{\bf x})=0\right\}\right)=0.$$
For ${\bf a}\in A_k^k$, define $$E_{\bf a}:=\left\{{\bf x}\in(0,1)^k:F_\psi({\bf a},{\bf x})\neq0\right\},$$ 
 and  the function $\Psi_{{\bf a}}:[0,1]^k\ra\bbr^k$ by 
$$(x_1,\dots,x_k)\mapsto \sum_{i=1}^ka_i\psi(x_i).$$
Further, define the measure
$$\nu_{\bf a}={\Psi_{{\bf a}}}_*\lambda_{[0,1]^k}.$$
 For fixed ${\bf a}\in A_k^k$, let $x_0\in E_a$. By the inverse function theorem, let   $\Omega\subset E_{\bf a}$ be a neighborhood of $x_0\in E_{\bf a}$ such that  $\Psi_{{\bf a}}$ is injective on $\Omega$.  Then ${\Psi_{{\bf a}}}_*(\lambda_{[0,1]^k}\big|_{\Omega})$ is an absolutely continuous measure with respect to $\lambda_{\bbr^k}$ and has Radon-Nikodym derivative $g_{{\bf a},\Omega}$, where 
\begin{equation}\label{radon}g_{{\bf a},\Omega}(x)=\begin{cases}\f{1}{F_\psi({\bf a},\left(\Psi_{{\bf a}}\big|_{\Omega}\right)^{-1}(x))}&x\in\Psi_{{\bf a}}(\Omega)\\0&otherwise.\end{cases}\end{equation}
Given a point $x\in\bbr$ and $\z>0$ let $I_\z(x)=[x-\z,x+\z]$. Given ${\bf x}\in\bbr^k$, let $I_\z({\bf x})=I_\z(x_1)\x\dots\x I_\z(x_k)$.  
\begin{lemma}\label{neigh}
Let ${\bf a}_0\in A_k^k$ and ${\bf x}_0\in E_{{\bf a}_0}:=\left\{{\bf x}\in[0,1]^k:F_\psi({\bf a}_0,{\bf x})\neq0\right\}$. There exist a  closed neighborhood $\Omega_{{\bf a}_0}$ of ${\bf a}_0$ in $A_k^k$ and $r_{{\bf a}_0}>0$  such that for every ${\bf a}\in \Omega_{{\bf a}_0}$, 
\begin{enumerate}
\item $I_{r_{{\bf a}_0}}({\bf x}_0)\sbq \left\{{\bf x}\in\bbr^k:F_\psi({\bf a},{\bf x})\neq0\right\}$
\item $\Psi_{\bf a}$ is injective on $I_{r_{{\bf a}_0}}({\bf x}_0)$.
\end{enumerate}
\end{lemma} 
\begin{proof}
Since ${\bf x}_0\in E_{{\bf a}_0}$, there exists a closed neighborhood $\Omega_{{\bf a}_0}$ of ${\bf a}_0$ in $A_k^k$ such that for every ${\bf a}\in\Omega_{a_0}$, ${\bf x}_0\in E_{\bf a}$. Define $r_1:\Omega_{{\bf a}_0}\ra\bbr$ as follows
$${\bf a}\mapsto \sup\left\{r>0: I_r({\bf x}_0)\sbq \left\{({\bf a},{\bf x})\in A_k^k\x\bbr^k:F_\psi({\bf a},{\bf x})\neq0\right\}\right\}.$$
As $F_\Psi$ is a continuous function in ${\bf a}$ and ${\bf x}$, the set $\left\{({\bf a},{\bf x})\in A_k^k\x[0,1]^k:F_\psi({\bf a},{\bf x})=0\right\}$ is a proper submanifold of $A_k^k\x [0,1]^k$. Thus $r_1$ is a continuous strictly positive function of ${\bf a}$.  Let $r_2:\omega_{{\bf a}_0}\ra\bbr^+$ be the lower bound of the injectivity radius given by the inverse function theorem. By construction $r_2$ is continuous and strictly positive on $\Omega_{{\bf a}_0}$. 
%
%
%
%
%
By continuity, 
 $r=\inf_{{\bf a}\in\Omega_{{\bf a}_0}}\{\min\{r_1({\bf a}),r_2({\bf a})\}\}$ satisfies 1. and 2.
\end{proof}

For a function $f\in C^m(\bbr^k)$,let
$$\|f\|_{C^m}:=\sup_{\al:|\al|\leq m}\|f^\al\|_\infty.$$

\begin{lemma}\label{bound} Further assume that $\psi$ as above is $C^{m+1}([0,1])$. Let $\Omega_A\sbq A_k^k$ be  compact and $I=I_1\x\dots\x I_k\sbq[0,1]^k$ where $I_i$ is a closed interval in $[0,1]$. Suppose that for ${\bf a}\in \Omega_A$, $\Psi_{\bf a}\Big|_{I}$ is injective. For ${\bf a}\in \Omega_A$ let $\nu_{{\bf a},I}= {\Psi_{\bf a}}_*\lambda_{[0,1]^k}\big|_{I}$.  Fix  $\e>0$. Then for every  ${\bf a}\in \Omega_A$, there exist positive measures measures $\tau_{1,{\bf a},I}$ and $\tau_{2,{\bf a},I}$ on $\bbr^k$ and a constant $C$, such that 
\begin{enumerate}
\item $\nu_{{\bf a},I}=\tau_{1,{\bf a},I}+\tau_{2,{\bf a},I}$
\item $ \tau_{1,{\bf a},I}=g_{{\bf a},I}\lambda_{\bbr^k}$ for some $g_{{\bf a},I}\in C^m(\bbr^k)$
\item $\|g_{{\bf a},I}\|_{C^m}<C$ 
\item $\tau_{2,{\bf a},I}(\bbr^k)<\e$.
\end{enumerate}
\end{lemma}
\begin{proof} By (\ref{radon}) 
$\nu_{{\bf a},I}=f_{{\bf a},I}\lambda_{\bbr^k}$, where
\begin{equation*}f_{{\bf a},I}(x)=\begin{cases}\f{1}{F({\bf a},\left(\Psi_{{\bf a}}\big|_{\Omega}\right)^{-1}({\bf x}))}&{\bf x}\in\Psi_{{\bf a}}(I)\\0&otherwise.\end{cases}\end{equation*}
By the assumption that $\psi\in C^{m+1}(\bbr)$, $f_{a,I}$ is piecewise $C^m(\bbr^k)$. Indeed, the only points where the derivatives are not continuous is at the boundary of $\Psi_{\bf a}(I)$. Furthermore, as $\Omega_A$ is compact, there exists $C_1$ such that for every ${\bf a}\in \Omega_A$, $x\in \Psi_{\bf a}(I)$ and multi-index $\al$ s.t. $|\al|\leq m$,
\begin{equation}\label{derivatives}|f_{{\bf a},I}^\al(x)|\leq C_1.\end{equation}
We obtain $g_{{\bf a},I}$ by multiplying $f_{{\bf a},I}$ by a $C^\infty$ function with values in $[0,1]$ which is $1$ on a large subset of $I$ and has support contained in the interior of $I$.
\end{proof}

\begin{remark}\label{strong} Observe that Lemmas \ref{neigh} and \ref{bound} are local statements. Indeed, neither requires that $\psi$ be differentiable on all of $[0,1]$. Lemma \ref{neigh} only requires differentiability on a neighborhood of ${\bf x}_0$ and Lemma \ref{bound} only requires the function to be differentiable on $I$. We need these additional constraints in section \ref{sev}.\end{remark}

\begin{corollary}\label{decomp} Again, let $B$ be a compact set in $A_k^k$ and suppose $\psi$ is $C^{m+1}$. Then given $\e>0$,  there exist $C\in\bbr^+$ and positive measures $\tau_{1,{\bf a}}$ and $\tau_{2,{\bf a}}$ on $\bbr^k$   for all ${\bf a}\in B$  such that 
\begin{enumerate}
\item $(\Psi_{\bf a})_*\lambda_{[0,1]^k}=\tau_{1,{\bf a}}+\tau_{2,{\bf a}}$
\item \label{item2}$\tau_{2,{\bf a}}(\bbr^k)<\epsilon$.
\item $\tau_{1,{\bf a}}=g_{\bf a}\lambda_{\bbr^k}$ for $g_{\bf a}$ in $C^m$ 
\item $\|g_{\bf a}\|_{C^m}<C$
\end{enumerate}
\end{corollary}
\begin{proof}
Let $\e>0$ be given. Let
$$E_{\bf a}^c:=\left\{{\bf x}\in[0,1]^k:F_\psi({\bf a},{\bf x})=0\right\},$$ 
and for $\delta>0$, let $$E_{{\bf a},\delta}^c:=\{x\in[0,1]^k:d(x,E_{\bf a}^c)<\delta\},$$
where $d(\cdot,\cdot)$ is the standard euclidean distance on $\bbr^k$.
 Let 
$h:B\x \bbr^+\ra[0,1]$ be defined as
$$({\bf a},\delta)\mapsto\lambda_{[0,1]^k}(E_{{\bf a},\delta}^c).$$
For every ${\bf a}\in B$, $E_{\bf a}^c$ is a proper Zariski closed subset of $[0,1]^k$. 
Hence for every fixed ${\bf a}\in B$, $$\lim_{\delta\ra0}H({\bf a},\delta)=0.$$
Furthermore, since $B$ is compact and $h$ is continuous in ${\bf a}$ and $\delta$, there exists a $\delta_0>0$ such that for all ${\bf a}\in B$, $$h({\bf a},\delta_0)<\e/2.$$
Define $$D_0:=\left\{({\bf a},{\bf x})\in A_k^k\x[0,1]^k:d(x,E_{\bf a}^c)\geq\delta \right\}.$$ Thus
\begin{equation}\label{more}\lambda_{[0,1]^k}\left({\bf x}\in[0,1]^k:({\bf a},{\bf x})\in D_0\right)>1-\e/2.\end{equation}
  For every $({\bf a},{\bf x})\in D_0$, let $\tilde{\Omega}_{({\bf a},{\bf x})}:=\tilde{\Omega}_{\bf a}\x I_{\tilde{r}({\bf a})}({\bf x})$, where $\tilde{\Omega}_{\bf a}\sbq B$,  a closed
  neighborhood of ${\bf a}$, and $\tilde{r}({\bf a})>0$ are found by Lemma \ref{neigh}. By 
compactness of $D_0$, there exist $({\bf a}_1,{\bf x}_1),\dots,({\bf a}_{l_0},{\bf x}_{l_0})$ such that $D_0\sbq \bigcup_{i=1}^{l_0}
\tilde{\Omega}_{(a_i,x_i)}\subseteq E$. Let $D_\e:=\bigcup_{i=1}^{l_0}
{\tilde{\Omega}_{(a_i,x_i)}}$. Since the union of finitely many rectangular sets can be decomposed into finitely many disjoint rectangular sets, there exist disjoint sets
$\Omega_1,\dots,\Omega_l$ such that $$\Omega_i=\Omega_{{\bf a}_i}\x \Omega_{{\bf x}_i}\sbq\tilde{\Omega}_{{\bf a}_j,{\bf x}_j}\text{ for some $j\in\{1,\dots,l_0\}$},$$ $\Omega_{{\bf x}_i}=I_{1,i}\x\dots\x I_{k,i}$ for closed intervals $I_{1,i},\dots, I_{k,i}$,
 and for every $ {\bf a}\in B$,
$$\{x\in[0,1]^k:({\bf a},x)\in D_\e\}=\cup_{i:a\in\Omega_{{\bf a}_i}}\Omega_{{\bf x}_i}.$$
 By  Lemma \ref{neigh}, the restriction of $\Psi_a$ to each $\Omega_{x_i}$ is injective.
For ${\bf a}\in 
B$, and $i$ such that ${\bf a}\in\Omega_{{\bf a}_i}$, let
$$\nu_{{\bf a},i}=(\Psi_{\bf a})_*\lambda_{[0,1]^k}\Big|_{\Omega_{{\bf x}_i}}.$$
Then by Lemma \ref{bound}, there exists $C_i$ such that  for all ${\bf a}\in \Omega_{{\bf a}_i}$,
$\nu_{{\bf a},i}$ is the sum of positive measures $g_{{\bf a},\Omega_{{\bf x}_i}}\lambda_{\bbr^k}$ and $\tau_{{\bf a},2,\Omega_{{\bf x}_i}}$ such that $\|g_{{\bf a},\Omega_{{\bf x}_i}}\|_{C^m}<C_i$ and $\tau_{{\bf a},2,\Omega_{{\bf x}_i}}(\bbr^k)<\e/2l$.
For ${\bf a}\in B$, let 
$$\tau_{{\bf a},1}=g_{\bf a}\lambda_{\bbr^k}:=\sum_{i:{\bf a}\in\Omega_{{\bf a}_i}}g_{{\bf a},\Omega_{{\bf x}_i}}\lambda_{\bbr^k},$$
and define
$\tau_{{\bf a},2}=(\Psi_{\bf a})_*\lambda_{[0,1]^k}-\tau_{{\bf a},1}.$  By construction $\|g_{\bf a}\|_{C^m}\leq\sum_{i=1}^l C_i$. Furthermore, by equation (\ref{more}) and the fact that $\tau_{{\bf a},2,\Omega_{x_i}}(\bbr^k)<\e/2l$ for each $i$ with ${\bf a}\in\Omega_{{\bf a}_i}$, we have that$\tau_{{\bf a},2}(\bbr^k)<\e$. 
\end{proof}
\section{Reformulation to a problem on $\bbr^k$}\label{three}
Recall the notation of Theorem \ref{mainthm}.  Let $\phi:[0,1]\ra\bbr^k$ be a totally non-planar $C^{m+1}$ function. Let $\mu_U=\mathpzc{u}_*\phi_*\lambda_{[0,1]}$, and $\mu_A$ be a  compactly supported asymptotically $U$-expanding measure on $A_{k_0}$. Let $\mu=\mu_A*(\mathpzc{u}_*\nu)$ and $\mu^{*n}$ be the $n$-fold convolution of $\mu$.

For a fixed $a\in A_{k_0}$, define $C_a\in A_k$ to be the unique diagonal matrix such that 
$a^{-1}\mathpzc{u}(x)a=\mathpzc{u}(C_a  x)$.
 Fix $n\in\bbn$ and let $f\in C_c(G)$ be given. Then
 \begin{eqnarray*}&&\int_{G}fd\mu^{*n}\\&=&\int_{A_{k_0}}\int_{\bbr^k}\dots\int_{A_{k_0}}\int_{\bbr^k}f\left(a_n u_n\ldots a_1  u_1\right) d\mu_U(u_1) d\mu_A(a_1)\dots d\mu_U(u_n) d\mu_A(a_n)\\
&=&\int_{A_{k_0}^n}\int_{[0,1]^n}f\left(\prod_{i=1}^na_i u \left(\prod_{i=1}^{n-1} C_{a_i}\phi(t_n)+\prod_{i=1}^{n-2} C_{a_i}\phi(t_{n-1})+\dots+\phi(t_1)\right) \right) d\lambda_{[0,1]^n}({\bf t})d\mu_{A^n}({\bf a})
\end{eqnarray*}
For $i\in\bbn$ define, the map $\theta_i:A_{k_0}^i\ra A_{k}$ as 
\begin{equation}\label{eqthetai}\theta_i(a_1,\dots,a_i)=\prod_{j=1}^i C_{a_j}.\end{equation}
Define the map $\theta:A_{k_0}^k\ra A_k^k$ as
$$\theta(a_1,\dots,a_k)=\big(\theta_{k-1}(a_1,\dots,a_{k-1}),\theta_{k-1}(a_1,\dots,a_{k-2}),\dots, \theta_1(a_1), e\big).$$
Define the map $\Pi:A_{k_0}^k\ra A_{k_0}$ as
$$\Pi(a_1,\dots,a_k)=\prod_{i=1}^k a_i.$$
For ${\bf a}\in A_k^k$, define $\Phi_{\bf a}:[0,1]^k\ra\bbr^k$ as
$$\Phi_{\bf a}(x_1,\dots,x_k):=\sum_{i=1}^ka_i\phi(x_i).$$
Then for $f\in C_c(G)$
\begin{equation}\label{mudecomp}\int_Gfd\mu^{*k}=\int_{A_{k_0}^k}\int_{[0,1]^k}f\left(\Pi({\bf a}) u\left(\Phi_{\theta({\bf a})}({\bf x})\right)\right)d\lambda_{[0,1]^k}({\bf x})d\mu_A^k({\bf a}).\end{equation}
For $n\in N$, define the map $\Pi_n:(A_{k_0}^k)^n\ra A_0$ as
$$\Pi_n({\bf a}_1,\dots,{\bf a}_n)=\prod_{i=1}^n \Pi({\bf a}_i).$$
With the above notation, for $n\in\bbn$ 
$$\int_Xfd\mu^{*k n}=\int_{(A_{k_0}^k)^n}\int_{([0,1]^k)^n}f\left(\Pi_n({\bf a}) u\left(\sum_{i=1}^n
\theta_{k(i-1)}({\bf a})\Phi_{\theta({\bf a}_i)}(x_i)\right)\right)d(\lambda_{[0,1]^k})^n({\bf x})d(\mu_A^k)^n({\bf a})$$
 For ${\bf a}\in A_{k_0}^{k }$, define 
$$\nu_{\bf a}:={\Phi_{\theta({\bf a})}}_*\lambda_{[0,1]^k},$$
and for ${\overline{\bf a}}=({\bf a}_1,\dots,{\bf a}_n)\in(A_{k_0}^k)^n$, 
define 
\begin{equation}\label{nu}\nu_{{\overline{\bf a}}}:=({\theta_{k(n-1)}}({\overline{\bf a}})_*\nu_{{\bf a}_n})*({\theta_{k(n-2)}}({\overline{\bf a}})_*\nu_{{\bf a}_{n-1}})*\dots*\nu_{{\bf a}_1},\end{equation}
where if ${\overline{\bf a}}=((a_{1,1},\dots,a_{1,k}),\dots,(a_{n,1},\dots,a_{n,k}))$, then $${\theta_{k(i)}}({\overline{\bf a}})={\theta_{ki}}(a_{1,1},a_{1,2},\dots,a_{i,k}).$$
Then 
\begin{equation}\label{measures}\int_Gfd\mu^{*k n}=\int_{(A_{k_0}^k)^n}\int_{\bbr^k}f\left(\Pi_n({\overline{\bf a}}) u\left({\bf x}\right)\right)d\nu_{\overline{\bf a}}({\bf x})d(\mu_A^k)^n({\overline{\bf a}}).\end{equation}
We will understand the measure $\mu^{*n}$ by  understanding the measures $\nu_{{\overline{\bf a}}}$ for ${\overline{\bf a}}\in (A_{k_0}^k)^n$ for large $n\in\bbn$. 
\begin{lemma}\label{norm}Let $\mu_1$ and $\mu_2$ be measures on $\bbr^k$. Suppose $\mu_1=g\lambda_{\bbr^k}$ where $g\in C^m$ with $\|g\|_{C^m}<\infty$, and $\mu_2(\bbr^k)=\delta<\infty$. Then $\mu_1*\mu_2=(g*\mu_2)\lambda_{\bbr^k}$ (here $g*\mu_2(x)=\int g(x-y) d\mu_2(y)$), and 
$$\|g*\mu_2\|_{C^m}\leq\delta\|g\|_{C^m}.$$
\end{lemma}
\begin{proof} Let $f\in C_c(\bbr^k)$ be given. Then 
\begin{eqnarray*}
\int f d(\mu_1*\mu_2)&=&\int\int f(x+y)d\mu_1(x)d\mu_2(y)\\
&=&\int\int f(x+y)g(x)d\lambda_{\bbr^k}(x)d\mu_2(y)\\
&=&\int f(x)\int g(x-y)d\mu_2(y)d\lambda_{\bbr^k}(x).
\end{eqnarray*} Let $|\alpha|\leq m$. Then by dominated convergence theorem,
\begin{eqnarray*}
\left\|\left(\int g(x-y)\mu_2(y)\right)^\alpha\right\|_\infty&=&\left\|\int g^\alpha(x-y)d\mu_2(y)\right\|_\infty\\
&\leq&\int \left\|g^\alpha(x)\right\|_\infty d\mu_2(y)\\
&=&\delta \left\|g^\alpha(x)\right\|_\infty.
\end{eqnarray*}
\end{proof}


\begin{lemma}\label{main} There exist $\eta>0$ and $C_1>0$ such that for all $n\in\bbn$, if ${\overline{\bf a}}\in (\supp(\mu_A)^k)^n$, then
there exist positive measures $\nu_{1,{\overline{\bf a}}}$ and $\nu_{2,{\overline{\bf a}}}$ on $\bbr^k$ such that 
\begin{enumerate}
\item $\nu_{{\overline{\bf a}}}=\nu_{1,{\overline{\bf a}}}+\nu_{2,{\overline{\bf a}}}$
\item $\nu_{1,{\overline{\bf a}}}=g_{\overline{\bf a}}\lambda_{\bbr^k}$ and $\|g_{\overline{\bf a}}\|_{C^m}<C_1$
\item\label{item3} $\nu_{2,{\overline{\bf a}}}(\bbr^k)<e^{-\eta n}$.
\end{enumerate}
\end{lemma}
\begin{proof}Let $n\in\bbn$ be given. As $\mu_A$ was compactly supported and the map $a\mapsto C_a$ is continuous, there exist $M_1<1<M_2$ such that if $a$ is in the support of $\mu_A^n$, then $\theta_n(a)$ is a diagonal matrix in $A_k$ with diagonal entries in $ [M_1^n,M_2^n]$.

Let $\e>0$ be such that $\e*M_1^{-(km+k^2)}<1$. Let $B=\supp(\mu_A)^k$. Then  since $C_a$ is a 
continuous function of $A$ and $\mu_A$ was compactly supported,  $\theta(B)\sbq A^k_k$ is a 
compact set.  By applying Corollary \ref{decomp}   to $\theta(B)$ with this choice of $\e$ there exist
a constant $C$ and positive measures  $\tau_{1,\theta({\bf a})}$, and $\tau_{2,\theta({\bf a})}$ for every ${\bf a}\in 
B$,  such that 
$\nu_{{\bf a}}=\tau_{1,\theta({\bf a})}+\tau_{2,\theta({\bf a})}$ where $\tau_{1,\theta({\bf a})}=g_{{\bf a}_i}\lambda_{\bbr^k}$,  $\|g_{{\bf a}}\|_{C^m}<C$, and $\tau_{2,\theta({\bf a})}(\bbr^k)<\e$.  

Let ${\overline{\bf a}}\in \theta(B)^n$. Then
$$\theta_{k(i-1)}({\overline{\bf a}})_*\nu_{{\bf a}_i}=\theta_{k(i-1)}({\overline{\bf a}})_*\tau_{1,{\bf a}_i}+\theta_{k(i-1)}({\overline{\bf a}})_*\tau_{2,{\bf a}_i}.$$
Again, $\theta_{k(i-1)}({\overline{\bf a}})_*\tau_{2,{\bf a}_i}(\bbr^k)<\e$, and  for every $f\in C_{c}(\bbr^k)$,
\begin{eqnarray*}
\int f({\bf x}) d\left(\theta_{k(i-1)}({\overline{\bf a}})_*\tau_{1,{\bf a}_i}\right)({\bf x})&=&\int f\big(\theta_{k(i-1)}({\overline{\bf a}}) x\big)g_{{\bf a}_i}(x) \,d\lambda_{\bbr^k}(x)\\
&=&\det\left(\theta_{k(i-1)}({\overline{\bf a}})^{-1}\right)\int f(x) g_{{\bf a}_i}\big(\theta_{k(i-1)}({\bf a})^{-1} x\big) d\lambda_{\bbr^k}\\
\end{eqnarray*}
Thus 
\begin{equation}\label{acont}\theta_{k(i-1)}({\overline{\bf a}})_*\tau_{1,{\bf a}_i}=\det\left(\theta_{k(i-1)}({\overline{\bf a}})^{-1}\right) g_{{\bf a}_i}\big(\theta_{k(i-1)}({\overline{\bf a}})^{-1} x\big)\lambda_{\bbr^k}.\end{equation}

Recall that $\theta_n({\overline{\bf a}})$ is a diagonal matrix in $A_k$ with diagonal entries in $ [M_1^n,M_2^n]$ So, 
$$\det\left(\theta_{k(i-1)}({\overline{\bf a}})^{-1}\right)\leq M_1^{-k^2 (i-1)}.$$
Furthermore, the operator norm of $\theta_{k(i-1)}({\overline{\bf a}})^{-1}$ is at most $M_1^{-k(i-1)}$, so for a multi-index $\al$ such that $|\al|=l\leq m$,

\begin{equation*}\Big\| g_{{\bf a}_i}^\al\big(\theta_{k(i-1)}({\overline{\bf a}})^{-1} x\big)\Big\|_{\infty}\leq M_1^{-(i-1)(k l)}\|g_{{\bf a}_i}\|_{C^m}.\end{equation*}
Combining these equations, we have that
\begin{equation}\label{growth1}\Big\|\det\left(\theta_{k(i-1)}({\overline{\bf a}})^{-1}\right) g_{{\bf a}_i}\big(\theta_{k(i-1)}({\overline{\bf a}})^{-1} x\big)\Big\|_{C^m}\leq M_1^{-(i-1)(km+k^2)}\|g_{{\bf a}_i}\|_{C^m}.\end{equation}

Now
\begin{eqnarray}\label{stru}
\nu_{\overline{\bf a}}&=&\big({\theta_{k(n-1)}}({\overline{\bf a}})_*\nu_{{\bf a}_n}\big)*\dots*\nu_{{\bf a}_1}\\
&=&\big({\theta_{k(n-1)}}({\overline{\bf a}})_*\nu_{{\bf a}_n}\big)*\dots*\big(\tau_{1,\theta({\bf a}_1)}+\tau_{2,\theta({\bf a}_1)}\big)\nonumber\\
&=&\big({\theta_{k(n-1)}}({\overline{\bf a}})_*\nu_{{\bf a}_n}\big)*\dots*\big({\theta_{k}}({\overline{\bf a}})_*\nu_{{\bf a}_2}\big)*\tau_{1,\theta({\bf a}_1)}\nonumber\\
&&+\big({\theta_{k(n-1)}}({\overline{\bf a}})_*\nu_{{\bf a}_n}\big)*\dots*\big({\theta_{2k}}({\overline{\bf a}})_*\nu_{{\bf a}_3}\big)*\big(\theta_{k}({\overline{\bf a}})_*\tau_{1,\theta({\bf a}_2)}\big)*\tau_{2,\theta({\bf a}_1)}\nonumber\\
&&+\big({\theta_{k(n-1)}}({\overline{\bf a}})_*\nu_{{\bf a}_n}\big)*\dots*\big({\theta_{3k}}({\overline{\bf a}})_*\nu_{{\bf a}_4}\big)*\big(\theta_{2k}({\overline{\bf a}})_*\tau_{1,\theta({\bf a}_3)}\big)*\big(\theta_{k}({\overline{\bf a}})_*\tau_{2,\theta({\bf a}_2)}\big)*\tau_{2,\theta({\bf a}_1)}\nonumber\\
&&\vdots\nonumber\\
&&+\big({\theta_{(n-1)k}}({\overline{\bf a}})_*\nu_{{\bf a}_n}\big)*\big(\theta_{(n-2)k}({\overline{\bf a}})_*\tau_{1,\theta({\bf a}_{n-1})}\big)*\big(\theta_{(n-3)k}({\overline{\bf a}})_*\tau_{2,\theta({\bf a}_{n-2})}\big)*\dots*\tau_{2,\theta({\bf a}_1)}\nonumber\\
&&+\big(\theta_{(n-1)k}({\overline{\bf a}})_*\tau_{1,\theta({\bf a}_{n})}\big)*\big(\theta_{(n-2)k}({\overline{\bf a}})_*\tau_{2,\theta({\bf a}_{n-1})}\big)*\dots*\tau_{2,\theta({\bf a}_1)}\nonumber\\
&&+\big(\theta_{(n-1)k}({\bf a})_*\tau_{2,\theta({\bf a}_{n})}\big)*\dots*\tau_{2,\theta({\bf a}_1)}\nonumber
\end{eqnarray}
Define $\nu_{2,{\overline{\bf a}}}=\big(\theta_{(n-1)k}({\overline{\bf a}})_*\tau_{2,\theta({\bf a}_{n})}\big)*\dots*\tau_{2,\theta({\bf a}_1)}$, and $\nu_{1,{\overline{\bf a}}}=\nu_{\overline{\bf a}}-\nu_{2,{\overline{\bf a}}}.$ By construction $\nu_{2,{\overline{\bf a}}}(\bbr^k)<\e^n$. Thus \ref{item3}. holds with $\eta=-\log(\e)$.
By (\ref{stru}), 
$$\nu_{1,{\overline{\bf a}}}=\sum_{i=1}^n\big({\theta_{k(n-1)}}({\overline{\bf a}})_*\nu_{{\bf a}_n}\big)*\dots*\big({\theta_{ik}}({\overline{\bf a}})_*\nu_{{\bf a}_{i+1}}\big)*\big(\theta_{(i-1)k}({\overline{\bf a}})_*\tau_{1,\theta({\bf a}_i)}\big)*\big(\theta_{(i-2)k}({\overline{\bf a}})_*\tau_{2,\theta({\bf a}_{i-1})}\big)*\dots*\tau_{2,\theta({\bf a}_1)}.$$
By (\ref{acont}),  (\ref{growth1}), and Lemma \ref{norm}, there exist functions $g_1,\dots,g_n\in C^l$ such that for $i=1,\dots n$, 
$$\big({\theta_{k(n-1)}}({\overline{\bf a}})_*\nu_{{\bf a}_n}\big)*\dots*\big({\theta_{ik}}({\overline{\bf a}})_*\nu_{{\bf a}_{i+1}}\big)*\big(\theta_{(i-1)k}({\overline{\bf a}})_*\tau_{1,\theta({\bf a}_i)}\big)*\big(\theta_{(i-2)k}({\overline{\bf a}})_*\tau_{2,\theta({\bf a}_{i-1})}\big)*\dots*\tau_{2,\theta({\bf a}_1)}=g_i\lambda_{\bbr^k},$$
and
$$\|g_i\|_{C^m}\leq M_1^{-(i-1)(km+k^2)}C\epsilon^{i-1}.$$
Observe that $g_{\overline{\bf a}}=\sum_{i=1}^ng_i$ is the Radon-Nikodym derivative of $\nu_{1,{\overline{\bf a}}}$, and as we chose $\e$ to satisfy $\epsilon M_1^{-(km+k^2)}<1$, we have
\begin{eqnarray}\label{CMbound}
\|g_{\overline{\bf a}}\|_{C^m}&\leq&\sum_{i=1}^n\|g_i\|_{C^m}\\
&\leq&\sum_{i=1}^nM_1^{-(i-1)(km+k^2)}C\epsilon^{i-1}\nonumber\\
&<&\f{C}{1-\e M_1^{-(km+k^2)}.}\nonumber
\end{eqnarray}
\end{proof}
\section{Mixing of Diagonal Action}\label{sec4}
For ${\bf{t}}=(t_1,\dots,t_{k_0})\in \bbr^{k_0}$ such that $$t_1,\dots,t_{k_0}\in\bbr\qquad\text{ and }\qquad\sum_{i=1}^{k_1} t_i=\sum_{j=1}^{k_2}t_{k_1+j},$$
define $$a_{\bf{t}}:=\diag(e^{t_1},\dots,e^{t_{k_1}},e^{-t_{k_1+1}},\dots,e^{-k_0}),$$
$$\lfloor{\bf{t}}\rfloor:=\min_{i=1,\dots,k_0}t_i,$$
and 
$$\llfloor {\bf{t}}\rrfloor:=\min_{\substack{i\in\{1,\dots,k_1\}\\j\in\{k_1+1,\dots,k_0\}}}t_i+t_j.$$
Fix a right-invariant metric $dist$ on $G$, giving rise to the corresponding metric on $X$. Let $H$ be a subgroup of $G$, and for $r>0$, let $B_H(r)$ denote the open ball centered at $e$ of radius $r$ in $H$ according to this metric. For a function $\psi$ on $X$ let
$$\|\psi\|_{\text{Lip}}:=\sup_{x,y\in X, x\neq y }\f{|\psi(x)-\psi(y)|}{dist(x,y)}.$$
With this notation Kleinbock and Margulis prove the following.
\begin{theorem}[\cite{KM} Theorem 1.3]\label{TKM} There exists ${\gamma}$ such that for any $f\in C^\infty_{\text{comp}}(U),$ $\psi
\in C^\infty_{\text{comp}}(X)$ and any compact $L\sbq X$, there exists $C=C(f,\psi,L)$ such that for all $z\in L$ and all ${\bf t}\in
\bbr^k$,
\begin{equation}\label{orig}\left|\int_U f(u)\psi(a_{\bf{t}}tuz)d\lambda_{U}-\int_Uf \int_X\psi\right|\leq Ce^{-\gamma \lfloor\bf{t}\rfloor}.\end{equation}
\end{theorem}
This theorem almost suffices to prove the result, but we require a stronger understanding of the constant $C$ in the above theorem and the flexibility to pick $f\in C^m_{\text{comp}}(U)$ for a fixed $m$ depending on $G$.
We achieve this through the following stronger version of Theorem  1.3 of \cite{KM}, which follows from Lemma 2.3 of their paper, convolving with an approximate identity, and a small additional argument, which we have provided below.  
 \begin{theorem}[\cite{KM}]\label{mix} There exist ${\gamma}$ and $m\in\bbn$  such that for any $C_1, C_2$, and $f\in C^m_{\text{comp}}(U)$ such that $\int|f|<C_2$ and $\|f\|_{C^m}<C_1$, then for
  $\psi
\in C^\infty_{\text{comp}}(X)$ and any compact $L\sbq X$, there exists $C=C(C_1,C_2,\|\psi\|_{C^m},\|\psi\|_{\text{Lip}},L)$ such that for all $z\in L$ and all ${\bf t}\in
\bbr^k$,
$$\left|\int_U f(u)\psi(a_{\bf{t}}tuz)d\lambda_{U}-\int_Uf \int_X\psi\right|\leq Ce^{-\gamma \llfloor\bf{t}\rrfloor}.$$
 \end{theorem}
 \begin{proof} Recall  Lemma 2.3 of \cite{KM}
 \begin{lemma}[ Lemma 2.3 \cite{KM}] Let $f\in C_{\text{comp}}^m(U)$, $0<r<r_0/2$ and $z\in X$ be such that
 \begin{enumerate}
 \item[(i)] $\supp f\subset B_U(r)$, and 
 \item[(ii)] $\pi_z$ is injective on $B_G(2r)$, where $\pi_z:G\ra X$ is defined as $g\mapsto gz$.
  \end{enumerate}
 Then for any $\psi\in C^m_{\text{comp}}(X)\oplus\bbr1_X$ with $\int_X\psi=0,$ there exists $E=E(\|\psi\|_{C^m},\|\psi\|_{\text{Lip}})$ such that for any $t\geq 0$ and $z\in X$ one has
 \begin{equation}\label{partexp}\left|\int_{U}f(u)\psi(a_{\bf t} u z)\right|\leq E\left(r\int_U|f|+r^{-(2m+N/2)}\|f\|_{C^m}e^{-\gamma t}\right),\end{equation}
where $\gamma$ and $m$  are found according to the special gap of the $G$ action on $L^2(X)$ and $N=k_1^2+k_2^2+k_1k_2-1$.
 \end{lemma}
 Note that this version differs slightly from the version in \cite{KM}. Indeed, the $C^m$ norm in equation (\ref{partexp}) could be replaced with the $(m,2)$-Solbolev norm and the space which $\psi $ lies in could be generalized to Lipschitz functions in  the $(m,2)$-Solbolev space. Furthermore, we have changed $E$ from being a function of merely $\psi$ to a function of the $C^m$ norm of $\psi$ and the Lipschitz norm of $\psi$. The final line on page 390 of \cite{KM}  and the second display on page 391  of \cite{KM} justify this strengthening of the definition of $E$.
 Then, observation of the second and fifth displays on page 395 of \cite{KM} finishes the proof that $C$ in Theorem \ref{TKM} can be calculated as a function of $\int_U|f|$, $\|f\|_{C^m}$, $\|\psi\|_{C^m}$, $\|\psi\|_{\text{Lip}},$ and $L$.
 
 To generalize the statement of Theorem \ref{TKM} to cover $f\in C^m_{\text{comp}}(X)$, one need only convolve $f$ with a smooth approximate identity as follows.
 Let $\Omega_n$ be a sequence of open neighborhoods of the identity in $U$ such that $\cap_{n=1}^\infty\Omega_n=\{e\}$ and let 
 $h_n\in C^\infty_{\text{comp}}$ be a sequence of positive functions  such that
  $\int h_n d\lambda_U=1$ and $\supp(h_n)\subseteq \Omega_n$.
  Then for any $f\in C_{\text{comp}}^m(U)$, $\psi\in C^\infty_{\text{comp}}(X)$, $u\in U$, ${\bf t}\in\bbr^k$, $z\in X$ and $n\in\bbn$, 
  $$f*h_n\in C^\infty_{\text{comp}}(U),\qquad\qquad\lim_{n\ra\infty}f*h_n(u)=f(u),$$ 
  $$\lim_{n\ra\infty}\int_Uf*h_n(u)d\lambda_U(u)=\int_Uf(u)d\lambda_U(u),\qquad\qquad
\lim_{n\ra\infty}\int_Uf*h_n(u)\phi(a_{\bf t}uz)d\lambda_U(u)=\int_Uf(x)\phi(a_{\bf t}uz)d\lambda_U(u),$$ 
  $$\|f*h_n\|_{C^m}\leq \|f\|_{C^m},\qquad
  and
 \qquad\int|f*h_n|d\lambda_U\leq\int|f|d\lambda_U.$$
Combining these allows one to generalize to $f\in C^m_{\text{comp}}(X)$.

Finally, one must show that $\lfloor{\bf t}\rfloor$ can be replaced by  $\llfloor{\bf t}\rrfloor$ in equation (\ref{orig}). For this improvement, 
observe that if ${\bf t}\in\bbr^k$ and  $\llfloor{\bf t}\rrfloor=t$ then there exist ${\bf t}_1,{\bf t}_2\in\bbr^k$ such that $a_{\bf t}=a_{{\bf t}_1}
a_{{\bf t}_2}$, and $\lfloor{\bf t}_2\rfloor=t$. Let $\psi$ be given and define $\tilde{\psi}(x)=\psi(a_{{\bf t}_1} x)$. Then 
$\|\tilde{\psi}\|_{C^m}=\|\psi\|_{C^m}$, and $\|\tilde{\psi}\|_{\text{Lip}}=\|\psi\|_{\text{Lip}}$. Thus the theorem holds by applying the strengthened form of Theorem \ref{TKM} we proved above to $\tilde{\psi}$.
 \end{proof}
 
 \begin{remark} The $m$ in Theorem \ref{mainthm} will be chosen so that $m-1$ satisfies Theorem \ref{mix}.
 \end{remark}
 \section{Large Deviations for products of i.i.d. Random variables}
  The following Lemma I learned from conversations with Benoist and Quint, and it appears in the literature as the Chernoff bound
 \begin{lemma}[\cite{C} Theorem 1]\label{large}Let $X_1, X_2,\dots$ be real valued i.i.d. random variables with $E(e^{t|X_i|})<\infty$ for some $t>0$. Let $\rho=E(X_i)$, and $S_n=X_1+\dots+X_n$. Let $\e>0$. Then there exists $\eta>0$ such that 
 $$P\left(\left|\f{S_n}{n}-\rho\right|>\e\right)<e^{-n\eta}.$$
 \end{lemma}
\begin{corollary}\label{exp} Let $X_1,X_2,\dots$ be i.i.d. positively valued random variables with 
$\rho=E\ln(X)<\infty$
and $E(e^{|\ln(X_i)|})<\infty$. Let $\rho_1<\rho<\rho_2$, and 
$S_n=X_1\cdot \ldots\cdot X_n$.  Then there exists $\eta>0$ such that 
$$P(S_n<e^{n\rho_1})<e^{-\eta  n},$$
and
$$P(S_n>e^{n\rho_2})<e^{-\eta  n}.$$
\end{corollary}
\iftoggle{paper}{ This follows from taking logarithms and applying Lemma \ref{large}.}

\iftoggle{arxive}{
\begin{proof} Let $Y_i=\ln(X_i)$ and $T_n=Y_1+\dots+Y_n$. Then $E(Y_i)=\mu$ and by assumption for $t\leq1$, we have that $E(e^{t|Y_i|})<\infty$. So by Lemma \ref{large} there exists $\eta>0$ such that
\begin{eqnarray*}P(S_n<e^{n\rho_1})&=&P(T_n<n\rho_1)\\
&=&e^{-\eta  n}.
\end{eqnarray*}
The second statement follows similarly.
\end{proof}}

\iftoggle{thesis}{
\begin{proof} Let $Y_i=\ln(X_i)$ and $T_n=Y_1+\dots+Y_n$. Then $E(Y_i)=\mu$ and by assumption for $t\leq1$, we have that $E(e^{t|Y_i|})<\infty$. So by Lemma \ref{large} there exists $\eta>0$ such that
\begin{eqnarray*}P(S_n<e^{n\rho_1})&=&P(T_n<n\rho_1)\\
&=&e^{-\eta  n}.
\end{eqnarray*}
The second statement follows similarly.
\end{proof}}

\begin{corollary} \label{growth}Let $\mu_A$ be the  asymptotically $U$-expanding compactly supported measure on $A$ given in Theorem \ref{mainthm}.
Recall that for $i=1,\dots, k_0$ we defined
$\al_i=E(\ln\pi_i(\mu_A))$.  For $i\in\{1,\dots,k_1\}$ and $j\in\{k_1+1,\dots,k_0\}$ let $\be_{i,j}=(\al_i-\al_j)/2$, and
$$E_n=\left\{{a}=\diag(e^{a_1},\dots,e^{a_{k_0}})\in A_{k_0}: 
\parbox{1.6 in}{For all $i\in\{1,\dots,k_1\}$ and \\$j\in\{k_1+1,\dots,k_0\}$ \\ $a_i-a_j>n\be_{i,j}$}
\right\}$$
Then there exists $\eta>0$ such that 
$$\mu_A^{*n}(E_n)>1-e^{-\eta n}.$$
\end{corollary}
\iftoggle{arxive}{\begin{proof}
It suffices to show for each for $i\in\{1,\dots,k_1\}$ and $j\in\{k_1+1,\dots,k_0\}$  that there exists an $\eta_{i,j}$ such that
$$\mu_{A}^{*n}\left(\left\{{ a}=\diag(a_1,\dots,a_{k_0}):a_i-\al_j>n\be_{i,j})\right\}\right)>1-e^{-\eta_{i,j}n}.$$
This follows from Corollary \ref{exp}.
\end{proof}}

\iftoggle{thesis}{\begin{proof}
It suffices to show for each for $i\in\{1,\dots,k_1\}$ and $j\in\{k_1+1,\dots,k_0\}$  that there exists an $\eta_{i,j}$ such that
$$\mu_{A}^{*n}\left(\left\{{ a}=\diag(a_1,\dots,a_{k_0}):a_i-\al_j>n\be_{i,j}\right\}\right)>1-e^{-\eta_{i,j}n}.$$
This follows from Corollary \ref{exp}.
\end{proof}}

Let us conclude this section by proving that the measures we consider are uniformly expanding as defined in equations (\ref{unifex1}) and (\ref{unifex2}). 
Note that, a reader concerned only with the proof of Theorem \ref{mainthm} could skip the remainder of this section, as the condition that the measures are uniformly expanding is not used in the proof of Theorem \ref{mainthm}. 

Let $\mathfrak{sl}(k_0,\bbr)$ denote the Lie algebra of $\Sl(k_0,\bbr)$ and $\mathfrak{u}$ denote the sub-Lie algebra of $\mathfrak{sl}(k_0,\bbr)$ corresponding to $U$. Let $Q$ denote the euclidean projection from $\mathfrak{sl}(k_0,\bbr)\subseteq\Mat_{k_0\x k_0}$ to $\mathfrak{u}$.
\begin{proposition} \label{ac}Suppose $\mu$ is an absolutely continuous measure on $\bbr^k$, and $v\neq 0$ is an arbitrary element of $\mathfrak{sl}(k_0,\bbr)$. Then  for $\mu$-a.e.  $x\in\bbr^k$, $$Q(\Ad(\mathpzc{u}(x))v)\neq0.$$
\end{proposition}
\begin{proof}Generally, if one has a unipotent group $W$ acting on a vector space $V$, $L$ is the subspace of $W$ fixed vectors and $P_L$ is the projection onto $L$, then for an $v\neq 0\in V$, $\{w\in W:P_L(wv)=0\}$ is a proper algebraic sub variety of $W$. In the case at hand the $U$ fixed vectors is exactly $Q(\mathfrak{sl}(k,\bbr))$, and $P_L=Q$. Then since the measure $\mathpzc{u}_*\mu$ is absolutely continuous with respect to the Haar measure on $U$, the claim follows.

\iftoggle{arxive}{ For a more concrete proof, 
let $v\neq 0\in\mathfrak{sl}(k_0,\bbr)$ be given. Let $A\in \Mat_{k_1\x k_1}$, $B\in \Mat_{k_1\x k_2}$, $C\in \Mat_{k_2\x k_1}$, and $D\in \Mat_{k_2\x k_2}$ be such that
$$v=\left[\begin{array}{cc}A&B\\C&D\end{array}\right].$$
Then for ${\bf x}\in\bbr^k\cong \Mat_{k_1\x k_2}$,
\begin{eqnarray*}
\Ad(\mathpzc{u}(x))v&=&\left[\begin{array}{cc}1&{\bf x}\\0&1\end{array}\right]\left[\begin{array}{cc}A&B\\C&D\end{array}\right]\left[\begin{array}{cc}1&-{\bf x}\\0&1\end{array}\right]\\
&=&\left[\begin{array}{cc}A+{\bf x}C&B-A{\bf x} +{\bf x}D-{\bf x}C{\bf x}\\C&D-C{\bf x}\end{array}\right]
\end{eqnarray*}
If $A,C,D$ are zero matricies then $B\neq0$ and $$Q\left(\Ad(\mathpzc{u}(x))v\right)=Q\left(\left[\begin{array}{cc}0&B\\0&0\end{array}\right]\right)=v\neq0\text{ for all ${\bf x}\in\bbr^k$.}$$
If $C\neq0$, then $$Q\left(\Ad(\mathpzc{u}(x))v\right)=\left[\begin{array}{cc}0&B-A{\bf x} +{\bf x}D-{\bf x}C{\bf x}\\0&0\end{array}\right]$$ is a nontrivial quadratic in $\bbr^k$. Hence is non-zero almost everywhere with respect to an absolutely continuous measure.

If $C=0$, then the $i,j$th entry of $B-A{\bf x} +{\bf x}D$  is
$$b_{ij}-\sum_{l=1}^{k_1}a_{il}{x}_{lj}+\sum_{l=1}^{k_2}{x}_{il}d_{lj}.$$
Thus for the claim to be false,  for $i\in\{1,\dots, k_1\}$ and $j\in\{1,\dots,k_2\}$,
$$g_{ij}({\bf x}):=b_{ij}+\sum_{l=1}^{k_1}a_{il}{x}_{lj}-\sum_{l=1}^{k_2}{x}_{il}d_{lj},$$
must be identically $0$. Hence $B=0$, $a_{ij}=0$ if $i\neq j$, $d_{ij}=0$ if $i\neq j$, and
$$a_{11}=\dots=a_{k_1k_1}=d_{11}=\dots=d_{k_2k_2}.$$
However, this implies $\trace(v)\neq0$ and $v\not\in\mathfrak{sl}(k_0,\bbr)$.}

\iftoggle{thesis}{ For a more concrete proof, 
let $v\neq 0\in\mathfrak{sl}(k_0,\bbr)$ be given. Let $A\in \Mat_{k_1\x k_1}$, $B\in \Mat_{k_1\x k_2}$, $C\in \Mat_{k_2\x k_1}$, and $D\in \Mat_{k_2\x k_2}$ be such that
$$v=\left[\begin{array}{cc}A&B\\C&D\end{array}\right].$$
Then for ${\bf x}\in\bbr^k\cong \Mat_{k_1\x k_2}$,
\begin{eqnarray*}
\Ad(\mathpzc{u}(x))v&=&\left[\begin{array}{cc}1&{\bf x}\\0&1\end{array}\right]\left[\begin{array}{cc}A&B\\C&D\end{array}\right]\left[\begin{array}{cc}1&-{\bf x}\\0&1\end{array}\right]\\
&=&\left[\begin{array}{cc}A+{\bf x}C&B-A{\bf x} +{\bf x}D-{\bf x}C{\bf x}\\C&D-C{\bf x}\end{array}\right]
\end{eqnarray*}
If $A,C,D$ are zero matricies then $B\neq0$ and $$Q\left(\Ad(\mathpzc{u}(x))v\right)=Q\left(\left[\begin{array}{cc}0&B\\0&0\end{array}\right]\right)=v\neq0\text{ for all ${\bf x}\in\bbr^k$.}$$
If $C\neq0$, then $$Q\left(\Ad(\mathpzc{u}(x))v\right)=\left[\begin{array}{cc}0&B-A{\bf x} +{\bf x}D-{\bf x}C{\bf x}\\0&0\end{array}\right]$$ is a nontrivial quadratic in $\bbr^k$. Hence is non-zero almost everywhere with respect to an absolutely continuous measure.

If $C=0$, then the $i,j$th entry of $B-A{\bf x} +{\bf x}D$  is
$$b_{ij}-\sum_{l=1}^{k_1}a_{il}{x}_{lj}+\sum_{l=1}^{k_2}{x}_{il}d_{lj}.$$
Thus for the claim to be false,  for $i\in\{1,\dots, k_1\}$ and $j\in\{1,\dots,k_2\}$,
$$g_{ij}({\bf x}):=b_{ij}+\sum_{l=1}^{k_1}a_{il}{x}_{lj}-\sum_{l=1}^{k_2}{x}_{il}d_{lj},$$
must be identically $0$. Hence $B=0$, $a_{ij}=0$ if $i\neq j$, $d_{ij}=0$ if $i\neq j$, and
$$a_{11}=\dots=a_{k_1k_1}=d_{11}=\dots=d_{k_2k_2}.$$
However, this implies $\trace(v)\neq0$ and $v\not\in\mathfrak{sl}(k_0,\bbr)$.}

\end{proof}

\begin{proposition} \label{expon}Fix $v\in\mathfrak{sl}(k_0,\bbr)$ such that $Q(v)\neq0$, then there exists $\eta>0$ and $\al>0$ for all $n\in\bbn$,
$$\mu_A^\bbn\left({\bf a}=(a_1,a_2,\dots)\in A_{k_0}^\bbn:\f1n\log\|\Ad(a_n\dots a_1)v\|<\alpha\right)<e^{-\eta n}.$$
\end{proposition}
\begin{proof} Fix $n\geq k$.  Recall that for $i=1,\dots,k_0$, we defined $\al_i=E(\ln\pi(\mu_A))$  and for $i\in\{1,\dots,k_1\}$ and $j\in\{k_1+1,\dots,k_0\}$  we defined $\be_{i,j}=(\al_i-\al_j)/2$. Let $\be_=\min\{|\be_{i,j}|\}$ and $\al=\be\ln\|v\|$.
Then for ${\bf a}\in A_{k_0}^n$ such that $\Pi({\bf a})\in E_n$,
$$\f1n\log\|\Ad(a_n\dots a_1)v\|\geq\f1n\log\left(e^{n\beta}\|v\|\right)=\alpha.$$
Thus the claim follows from Corollary \ref{growth}.
\end{proof}
\begin{proposition}\label{ue} The measure $\mu=\mu_A*\mu_U$ where $\mu_A$ is uniformly expanding and $\mu_U=\mathpzc{u}_*\phi_*\lambda_{[0,1]}$ for a totally non-planar function $\phi:[0,1]\ra\bbr^k$ is a $C^1$ uniformly expanding function. Note that, we only require $\phi$ to be $C^1$ here.
\end{proposition}
\begin{proof}Let us adjust the notation from Section \ref{three} for this proof.
For $n\in\bbn$, define the map $\theta:A_{k_0}^n\ra A_k^n$ as
$$\theta(a_1,\dots,a_n)=\big(\theta_{n-1}(a_1,\dots,a_{n-1}),\theta_{n-1}(a_1,\dots,a_{n-2}),\dots, \theta_1(a_1), e\big),$$
where $\theta_i$ is as defined in Equation (\ref{eqthetai}).
 For ${\bf a}\in A_k^n$, define $\Phi_{n,{\bf a}}:[0,1]^n\ra\bbr^k$ as
$$\Phi_{n,{\bf a}}(x_1,\dots,x_n):=\sum_{i=1}^na_i\phi(x_i).$$
As shown in equation (\ref{mudecomp}), for $n\in\bbn$ and measurable function $f$ on $G$,
\begin{equation*}\int_Gfd\mu^{*n}=\int_{A_{k_0}^n}\int_{[0,1]^n}f\left(\Pi({\bf a}) u\left(\Phi_{n,\theta({\bf a})}({\bf x})\right)\right)d\lambda_{[0,1]^n}({\bf x})d\mu_A^n({\bf a}).\end{equation*}
For $n\geq k$ and all ${\bf a}\in A_{k_0}^n$, $$(\Phi_{n, \theta({\bf a})})_*\lambda_{[0,1]^n}=\tau_n*(\Phi_{k,\theta({\bf a})})_*\lambda_{[0,1]^k},$$ for some probability measure $\tau_n$.
By Lemma \ref{abscon}, $(\Phi_{k,\theta({\bf a})})_*\lambda_{[0,1]^k}$  is absolutely continuous with respect to Lebesgue measure on $\bbr^k$. Thus for $n\geq k$,
\begin{equation}\label{decomp2}\mu^{*n}=\int\Pi({\bf a})_*\mathpzc{u}_*\nu_{\bf a} d\mu_{A}^n({\bf a}),\end{equation}
where $\nu_{\bf a}$ is an absolutely continuous measure on $\bbr^k$.

Fix $v\neq0\in\mathfrak{sl}(k,\bbr)$. By combining Propositions \ref{ac}  with equation(\ref{decomp2}), it follows that for  $\mu^n$ almost every  $(g_1,\dots,g_n)\in G^n$, that
$$\Ad(g_n\dots g_1)v=\Ad(a)\Ad(u)v,\text{ with $Q(\Ad(u)v)\neq0$}$$
for some $a\in A_{k_0}^k$ and $u\in U$. Thus by Proposition \ref{ue}
 there exist a $\al>0$ and an $\eta>0$ such that for $n\geq k$,
$$\mu^{n}\left({\bf g}=(g_1,g_2,\dots)\in G^\bbn:\f1n\log\|\Ad(g_n\dots g_1)v\|<\alpha\right)<e^{-\eta n}.$$
Hence by the Borel Cantelli lemma,
$$\mu^\bbn\left({\bf g}=(g_1,g_2,\dots)\in G^\bbn:\f1n\log\|\Ad(g_n\dots g_1)v\|<\alpha\text{ for infinitely many $n$ }\right)=0.$$
Thus almost surely with respect to $\mu^\bbn$ and every $v\neq0\in\mathfrak{sl}(k_0,\bbr)$,
$$\lim_{n\ra\infty}\f1n\log\|\Ad(g_n\dots g_1)v\|>0.$$
\end{proof}
\section{Proof of Theorem \ref{mainthm}}
This section combines Lemma \ref{main}, the exponential mixing results of \cite{KM}, and Corollary \ref{exp} to prove Theorem \ref{mainthm}. 
\begin{proof} Let $L$ be a compact subset of $X$. Choose $\eta_1>0$ and $m\in\bbn$ to satisfy Theorem \ref{mix} and assume that $\phi$ is $C^{m+1}$. Fix $f\in C^{\infty}_c(X)$,  $z\in L$, and $n\in\bbn$. Then by (\ref{measures})
\begin{equation*}\int_Xfd(\mu^{*k n}*\delta_z)=\int_{(A_{k_0}^k)^n}\int_{\bbr^k}f\left(\Pi_n({\overline{\bf a}}) \mathpzc{u}({\bf x}) z\right)d\nu_{\overline{\bf a}}({\bf x})d(\mu_A^k)^n({\overline{\bf a}}).\end{equation*}
By applying Lemma \ref{main} to each ${\overline{\bf a}}\in (\sup (\mu_A)^k)^n$, there exist $\eta_2$, $C_1>0$, and $\nu_{\overline{\bf a}}=g_{\overline{\bf a}}\lambda_{\bbr^k}+\nu_{2,{\overline{\bf a}}}$ 
 such that  $\nu_{2,{\overline{\bf a}}}(\bbr^k)<e^{-\eta_2 n}$, and $\|g_{\overline{\bf a}}\|_{C^m}<C_1$. Hence
\begin{eqnarray}\label{1stmix}\int_Gfd(\mu^{*k n}*\delta_z)&=&\int_{(A_{k_0}^k)^n}\int_{\bbr^k}g_{\overline{\bf a}}({\bf x})f\left(\Pi_n({\overline{\bf a}}) \mathpzc{u}({\bf x}) z\right)d\lambda_{\bbr^k}({\bf x})d(\mu_A^k)^n({\overline{\bf a}})\\
&&+\int_{(A_{k_0}^k)^n}\int_{\bbr^k}f\left(\Pi_n({\overline{\bf a}}) \mathpzc{u}({\bf x}) z\right)d\nu_{2,{\overline{\bf a}}}({\bf x})d(\mu_A^k)^n({\overline{\bf a}}).\nonumber\end{eqnarray}
For ${\overline{\bf a}}\in (\sup (\mu_A)^k)^n$ such that $\Pi({\overline{\bf a}})\in E_{n k}$, we will apply Theorem \ref{mix} to 
\begin{equation*}\label{goal} \int_{\bbr^k}g_{\overline{\bf a}}({\bf x})f\left(\Pi_n({\overline{\bf a}}) \mathpzc{u}({\bf x}) z\right)d\lambda_{\bbr^k}({\bf x}).\end{equation*}
For such  ${\overline{\bf a}}\in E_{n k}$, $\Pi_n({\overline{\bf a}})=a_{\bf{t}}$, with ${\bf{t}}=(t_1,\dots,t_{k_0})$,  $t_i-t_j>n(\al_i-\al_j)/2$ for $i\in\{1,\dots,k_1\}$ and $j\in\{k_1+1,\dots,k_0\}$, and $t_i<n\beta_i$ for $i=k_1+1,\dots,k_0$. Recall that $\be=\min_{i=1,\dots,k_0}\{|\be_i|\}$. Then
$\llfloor{\bf{t}}\rrfloor>nk\be$. 
By construction $g_{\overline{\bf a}}$ is positive and $\int| g_{\overline{\bf a}}|\,d\lambda_{\bbr^k}<1$. Thus with $C=C(C_1,1,f,L)$ satisfying Theorem \ref{mix}, for ${\overline{\bf a}}\in (\sup (\mu_A)^k)^n$ such that $\Pi({\overline{\bf a}})\in E_{n k}$
\begin{eqnarray}\left|\int_{\bbr^k}g_{\overline{\bf a}}({\bf x})f\left(\Pi_n({\overline{\bf a}}) \mathpzc{u}({\bf x}) z\right)d\lambda_{\bbr^k}({\bf x})-\int_{\bbr^k}g_{\overline{\bf a}}\int_Xf\right|&=&\left|\int_{\bbr^k}g_{\overline{\bf a}}({\bf x})f\left(a_{\bf{t}} \mathpzc{u}({\bf x}) z\right)d\lambda_{\bbr^k}({\bf x})-\int_{\bbr^k}g_{\overline{\bf a}}\int_Xf\right|\nonumber\\
\label{mixing}&\leq&Ce^{-\eta_1nk\be}.
\end{eqnarray}
By Corollary \ref{growth}  there exists $\eta_3>0$ such that
\begin{equation}\label{sizee}(\mu_{A}^k)^n(E_{k n})>1-e^{\eta_3k n}.\end{equation}
Thus by combining (\ref{1stmix}),  (\ref{mixing}), and (\ref{sizee}),
\begin{eqnarray*}
\left|\int_Xfd(\mu^{*k n}*\delta_z)-\int_X f\right|
&\leq&\left|\int_{(A_{k_0}^k)^n}\int_{\bbr^k}g_{\overline{\bf a}}({\bf x})f\left(\Pi_n({\overline{\bf a}}) \mathpzc{u}({\bf x}) z\right)d\lambda_{\bbr^k}({\bf x})d(\mu_A^k)^n({\overline{\bf a}})-\int_X f\right|\\
&&+\left|\int_{(A_{k_0}^k)^n}\int_{\bbr^k}f\left(\Pi_n({\overline{\bf a}}) \mathpzc{u}({\bf x}) z\right)d\nu_{2,{\overline{\bf a}}}({\bf x})d(\mu_A^k)^n({\overline{\bf a}})\right|\\
&\leq&\left|\int_{E_{n k}}\left[\int_{\bbr^k}g_{\overline{\bf a}}({\bf x})f\left(\Pi_n({\overline{\bf a}}) \mathpzc{u}({\bf x}) z\right)d\lambda_{\bbr^k}({\bf x})-\int_{\bbr^k}g_{\overline{\bf a}}\int_Xf \right]d(\mu_A^k)^n({\overline{\bf a}})\right|\\
&&+\int_{{\overline{\bf a}}: \Pi({\overline{\bf a}})\in E_{n k}}\left|\int_{\bbr^k}g_{\overline{\bf a}}d\lambda_{\bbr^k}({\bf x})\int_Xf-\int_Xf \right|d(\mu_A^k)^n({\overline{\bf a}})\\
&&+\int_{{\overline{\bf a}}: \Pi({\overline{\bf a}})\in E_{n k}^c}\left[\int_{\bbr^k}g_{\overline{\bf a}}({\bf x})\|f\|_\infty d\lambda_{\bbr^k}({\bf x})+\left|\int_Xf \right|\right]d(\mu_A^k)^n({\overline{\bf a}})\\
&&+\|f\|_\infty e^{-\eta_2 n}\\
&\leq&C_2e^{-\eta_1n\beta}+e^{-\eta_2 n}\left|\int_Xf \right|+e^{-\eta_3k n}\left(\|f\|_\infty+\left|\int_Xf \right|\right)+\|f\|_\infty e^{-\eta_2 n}.
\end{eqnarray*}
\end{proof}
\section{Proof of Corollary \ref{mainthm3}} 

Corollary \ref{mainthm3} is a result of the following more general lemma:
\begin{lemma} Let $H$ be a group acting on a homogeneous  space $Y$ with $H$-invariant measure $\lambda_Y$. Let $\mu$ be a measure on $H$.  Suppose for any compact set $L\subseteq Y$, $f\in C_c(Y)$, and $\e>0$, that there exists an $N\in\bbn$, such that for any $z\in L$ and $n>N$,
\begin{equation}\label{conv}\left|\int fd(\mu^{*n}*\delta_z)-\int f\,d\lambda_Y\right|<\e.\end{equation}
Let $g_1,g_2,\dots$ be $i.i.d$ with law $\mu$. Then for every $z\in Y$ and $f\in C_c(Y)$, for almost every $(g_1,g_2,\dots)$ with respect to $\mu^{\otimes\bbn}$
$$\f1n\sum_{i=1}^nf(g_i\dots g_1z)\ra\int f\, d\lambda_Y.$$
\end{lemma}
\begin{proof}
Let $f$ be a continuous function on $Y$ such that $\int f d\lambda_X=0$ and let $z\in X$ be given. Fix $\e>0$ be given and fix a compact set $L$ such that $\lambda_X(L)>1-\e$ and $z\in L$.  By simultaneously applying (\ref{conv}) to $f$ and a function $\psi\in C_c(Y)$ which takes values in $[0,1]$ and is $1$ on $L$, there exists $N\in\bbn$ such that for $n\geq N$ and $y\in L$,  $$\mu^{*n}*\delta_y(L^c)<\e,$$
and
$$\left|\int f(gy)d\mu^{*n}(g)-\int f d\lambda_X\right|<\e.$$
Let $\nu=\mu^{\otimes\bbn}$ be the product measure on $H^{\otimes\bbn}$. Then for $n>2N\in \bbn$,
\begin{eqnarray*}
\int\left(\f1n\sum_{i=1}^nf(g_i\dots g_1z)\right)^2\,d\nu
&\leq&\f{1}{n^2}3N n\|f\|_\infty^2\\
&&+ \f{2}{n^2}\left|\sum_{\substack{N\leq i<j\leq n:\\|i-j|\geq N}}\int 1_L(g_i\dots g_1z)
f(g_i\dots g_1z) f(g_j\dots g_1z)d\nu\right|\\
&&+ \f{1}{n^2}\left|\sum_{\substack{N\leq i,j\leq n:\\|i-j|\geq N}}\int1_{L^c}(g_i\dots g_1z)
f(g_i\dots g_1z) f(g_j\dots g_1z)d\nu\right|\\
&\leq&\f{3N n}{n^2}\|f\|_\infty^2+ \f{2}{n^2}\sum_{\substack{N\leq i<j\leq n:\\|i-j|\geq N}}\|f\|_\infty\e+\f{1}{n^2}\sum_{\substack{N\leq i,j\leq n:\\|i-j|\geq N}}\mu^{*i}*\delta_y(L^c)\|f\|_\infty^2.
\end{eqnarray*}
As $\e$ was arbitrary as $n\ra\infty$, 
$$\int\left(\f1n\sum_{i=1}^nf(g_i\dots g_1z)\right)^2\,d\nu\ra 0.$$
Let $\delta>0$. By Chebyshev's inequality for every $n\in\bbn$,
\begin{equation*}
P\left(\f1n\left|\sum_{i=1}^nf(g_i\dots g_1z)\right|>\de\right)\leq\f{1}{\de^2}\int\left(\f1n\sum_{i=1}^nf(g_i\dots g_1z)\right)^2\,d\nu.
\end{equation*}
Since the right hand side goes to 0 for every $\de>0$, the claim follows.
\end{proof}
\section{Proof of Theorem \ref{mainthm2}}\label{sev}
Recall that in this case, $\phi:[0,1]\ra\bbr^k$ was a piecewise $C^{m+1}$ function and $x_1,\dots,x_k\in[0,1]$ were such that $\rspan\{\phi'(x_1),\dots,\phi'(x_k)\}=\bbr^k$. We fixed $0<c<1$ and a second arbitrary compactly supported probability  measure $\mu_R$ on $\bbr^k$. Then define $$\mu_U=\mathpzc{u}_*(c\phi_*\lambda_{[0,1]}+ (1-c)\mu_R).$$ 
Since $\phi$ is piecewise differentiable, for each $i\in\{1,\dots,k\}$ there exists a neighborhood $I_i$ of $x_i$ such that $\phi$ is differentiable on $I_i$. Define $F_\phi: A_k^k\x I_1\x\dots\x I_k\ra\bbr^k$ as
$$(a_1,\dots,a_k,y_1,\dots,y_k)\mapsto \det\left[a_1\psi'(y_1),\dots,a_k\psi(y_k)\right].$$
Let ${\bf x}=(x_1,\dots,x_k)\in[0,1]^k$ where $x_1,\dots,k_k$ are the fixed points in the problem statement. Recall remark \ref{strong}. The proof of  Lemma \ref{neigh} only required differentiability in a neighborhood of ${\bf x}_0$. Thus there exist a neighborhood $
\Omega_{e}$ of ${\bf e}\in A^k_k$, and $r>0$ such that $F_\phi({\bf a},{\bf y})$ is bounded away from $0$ on $
\Omega_e\x I_r({\bf x})$, and for ${\bf a}\in\Omega_e$, $\Phi_{\bf a}({\bf y}):=\sum_{i=1}^ka_i\phi(y_i)$ is injective on $I_r({\bf x})$. Thus by remark \ref{strong} and Lemma \ref{bound}, there exist $C>0$ and $g_{\bf a}\in C^m$ and $\tau_{{\bf a},2}$ for all ${\bf a}\in\Omega_e$ such that
$$\Phi_a(\lambda_{[0,1]^k}\big|_{I_r({\bf x})})=g_{\bf a}\lambda_{\bbr^k}+\tau_{{\bf a},2},$$
$\|g_{\bf a}\|_{C^m}<C$, and $\tau_{2,{\bf a}}<\delta$, where we choose $\delta=\f12\lambda_{[0,1]}(I_r({\bf x}))$.
Hence $g_{\bf a}\lambda_{\bbr^k}(\bbr^k)>\delta$.

Recall that for a fixed $a\in A_{k_0}$, define $C_a\in A_k$ to be the unique diagonal matrix such that 
$a^{-1}\mathpzc{u}(x)a=\mathpzc{u}(C_a  x)$. Pick $\e>0$ such that for every $a\in A_{k_0,\e}$, we have that $C_a\in \Omega_{e}$, $C_a\in A_{k,2\e}$, and $(e^{2\e})^{k^2+km}\delta<1$.   For ${\overline{\bf a}}\in (A_{k_0,\e}^k)^n$ define 
$\nu_{\overline{\bf a}}$ as in (\ref{nu}).  Recall (\ref{stru}). By applying this same decomposition to $\nu_{\overline{\bf a}}$, we can write $\nu_{\overline{\bf a}}$ as a sum of an absolutely continuous measure $g_{\overline{\bf a}} \,\lambda_{\bbr^k}$ and a second measure $\tau_{2,{\overline{\bf a}}}$ such that $\tau_{2,{\overline{\bf a}}}(\bbr^k)<\delta^n$. Furthermore, since we chose $(e^{2\e})^{k^2+km}\delta<1$,  by arguing as in (\ref{CMbound}), we can uniformly bounded $\|g_{\overline{\bf a}}\|_{C^m}$.
Thus we have the following version of Lemma \ref{main} in this case:

\begin{lemma} There exist $\eta>0$ and $C_1>0$ such that for all $n\in\bbn$ if ${\overline{\bf a}}\in( A_{k_0,\e}^k)^n$, then there exist positive measures $\nu_{1,{\overline{\bf a}}}$ and $\nu_{2,{\overline{\bf a}}}$ on $\bbr^k$ such that 
\begin{enumerate}
\item $\nu_{\overline{\bf a}}=\nu_{1,{\overline{\bf a}}}+\nu_{2,{\overline{\bf a}}}$
\item $\nu_{1,{\overline{\bf a}}}=g_{\overline{\bf a}}\lambda_{\bbr^k}$ and $\|g_{\overline{\bf a}}\|_{C^m}<C_1$
\item\label{item3} $\nu_{2,{\overline{\bf a}}}(\bbr^k)<e^{-\eta n}$.
\end{enumerate}
\end{lemma}
With this Lemma, by arguing as in  the proof of Theorem \ref{mainthm}, the proof of Theorem \ref{mainthm2} follows .

\end{document}